\documentclass[11pt]{article}
\usepackage{amssymb,amsmath,amsthm}
\usepackage{mathrsfs}
\usepackage[colorlinks=true,breaklinks]{hyperref}
\usepackage{xcolor}
\definecolor{c1}{rgb}{0,0,1} 
\definecolor{c2}{rgb}{0,0.3,0.9} 
\definecolor{c3}{rgb}{0.3,0,0.7} 
\hypersetup{
    linkcolor={c1}, 
    urlcolor={c2}, 
    citecolor={c3} 
}
\newcommand{\R}{\mathbb{R}}

\newcommand{\N}{\mathbb{N}}

\newtheorem{theorem}{Theorem}[section]

\newtheorem{lemma}[theorem]{Lemma}

\theoremstyle{definition}
\newtheorem{definition}[theorem]{Definition}

\newtheorem{remark}[theorem]{Remark}

\numberwithin{equation}{section}

\DeclareMathOperator{\meas}{meas}

\newcommand{\e}{\varepsilon}

\begin{document}


\title{Existence of infinitely many solutions for a class of fractional Schr\"{o}dinger equations in $\R^N$ with combined nonlinearities}
\date{}
\date{}
\author{Sofiane Khoutir\thanks{{E-mail: skhoutir@usthb.dz}}\\
{\small Faculty of Mathematics, Laboratory AMNEDP, University of Science and Technology Houari Boumediene,}\\
{\small PB 32 El-Alia, Bab Ezzouar 16111, Algiers, Algeria}}
\maketitle
 \begin{center}
 \begin{minipage}{13cm}
 \par
   \small  {\bf Abstract:} This paper is devoted to the following class of nonlinear fractional Schr\"{o}dinger equations:
\begin{equation*}
(-\Delta)^{s}u+V(x)u=f(x,u)+\lambda g(x,u), \quad \text{in}\: \R^N,
\end{equation*}
where $s\in (0,1)$, $\ N>2s$, $(-\Delta)^{s}$ stands for the fractional Laplacian, $\lambda\in \R$ is a parameter, $V\in C(\R^N,R)$, $f(x,u)$ is superlinear and $g(x,u)$ is sublinear with respect to $u$, respectively. We prove the existence of infinitely many high energy solutions of the aforementioned equation by means of the Fountain theorem. Some recent results are extended and sharply improved.
\end{minipage}
\end{center}

 \vskip6mm
\vskip2mm
 \par
\section{Introduction}
Consider the following fractional Schr\"{o}dinger equation
\begin{equation}\label{Schrod-frac}
(-\Delta )^{s}u+ V(x)u=f(x,u), \quad x \in \R^N,
\end{equation}
where $s \in (0,1)$, $N>2s$ and $(-\Delta )^{s}$ stands for the fractional Laplacian which can be defined for a sufficiently smooth function $u$ as
\begin{equation}\label{frac-Lap}
(-\Delta )^{s}u (x)=C(N,s)\lim\limits_{\e \rightarrow 0^+}\int_{\R^N\setminus B(x,\e)}\frac{u (x)-u (y)}{| x-y| ^{N+2s}}dy,\quad x\in \R^N,
\end{equation}
where $B(x,\e)=\{x\in \R^N\: :\: |x|< \e\}$ and $C(N,s)>0$ is a dimensional constant that depends on $N$ and $s$ (see \cite{29}). 

The Equation \eqref{1} arises in the study of the following fractional Schr\"{o}dinger equation
\begin{equation*}
i\partial_t \Psi(x,t)+(-\Delta )^{s}\Psi (x,t)+(V(x)-\omega)\Psi (x,t)=h(|\Psi (x,t)|)\Psi (x,t),
\end{equation*}
when looking for standing waves, that is, solutions of the form $\Psi (x,t)=\exp(-i\omega t)u(x)$.
The fractional Schr\"{o}dinger equation was introduced by Laskin \cite{Laskin1,Laskin2} in the context of fractional quantum mechanics, as a result of extending the Feynman path integral from the Brownian-like to L\'{e}vy-like quantum mechanical paths. It is also appeared in several subjects such as plasma physics, image processing, finance and stochastic models, see for instance \cite{finance,plasma,image,Laskin3}.

In recent years, the equation \eqref{Schrod-frac} has been extensively studied under various assumptions on $V$ and $f$ and there are many interesting results in the literature on the existence and multiplicity of solutions to problem (\ref{1}) has been obtained via variational approaches, we refer the readers to \cite{secchi1,19,18,felmer2,17,hou,khoutir1,khoutir2,24,zupei,teng,mohsen,20,26}. 
In particular, the existence of infinitely many high or small energy solutions to problem \eqref{Schrod-frac} was established in \cite{19,18,17,hou,khoutir1,khoutir2,24,teng} by the aid of variant fountain theorems (see \cite{zou}) or the symmetric mountain pass theorem (see \cite{Willem}). However, there are few papers concern with the existence of infinitely many (high or small) energy solutions to problem \eqref{Schrod-frac} in the case where $f(x,u)$ is a combination of sublinear and superlinear terms at infinity with respect to $u$, see for instance \cite{18,24,mohsen}.

In \cite{18}, Du and Tian considered the following class of fractional Schr\"{o}dinger equation with concave and critical nonlinearities
\begin{equation}\label{du}
(-\Delta)^{s} u+V(x) u=\mu a(x)|u|^{q-2} u+|u|^{2_{s}^{*}-2} u, \quad x \in \mathbb{R}^{N},
\end{equation}
where (and in the sequel) $2_s^*=\frac{2N}{N-2s}$ is the critical Sobolev exponent, $\mu>0$ is a parameter, $1<q<2$, $a(x)$ is positive continuous functions satisfying $a(x) \in L^{\frac{2}{2-q}}\left(\mathbb{R}^{N}\right) \cap L^{\frac{2_{s}^{2}}{2 s}-q}\left(\mathbb{R}^{N}\right)$ and $V(x)$ satisfies the following assumptions
\begin{itemize}
\item[$(V)$] $V \in C(\R^N,\R)$ satisfies $\inf_{x\in \mathbb{R}^{N}}  V(x)\geq V_{0}>0$, where $V_{0}$ is a constant. Moreover, there exists $r_0>0$ such that
\begin{equation*}
\lim_{|y|\rightarrow \infty} \meas\{x \in \mathbb{R}^{N} \: : \:|x-y|\leq r_0,\: V(x)\leq M\}=0,\quad \forall M>0,
\end{equation*}
\end{itemize}
where $\meas(.)$ is the Lebesgue measure on $\R^N$. The authors proved that there exists $\mu^*> 0$ such that, for any $0 < \mu < \mu^*$, problem \eqref{du} possesses
infinitely many small energy solutions by using the Dual fountain theorem.

In \cite{mohsen}, Timoumi established infinitely many small energy solutions to the problem
\begin{equation*} 
(-\Delta )^{s}u+ V(x)u=g(x,u)+h(x,u), \quad \text{in}\: \R^N, 
\end{equation*}
by means of the Dual Fountain Theorem (see \cite{Willem}), where $V(x)$ satisfies assumptions $(V)$, $g(x,u)$ is sublinear in $u$ and $h(x,u)$ is superlinear in $u$.

Li and Shang \cite{24} studied the following problem
\begin{equation}\label{li}
(-\Delta)^{s} u+V(x) u=f(x, u)+\lambda h(x)|u|^{p-2} u,\quad x \in R^{N}
\end{equation}
where  $\lambda>0$ is a parameter, $p\in [1,2)$, $h \in L^{\frac{2}{2-p}}\left(R^{N}\right)$ and $V$ and $f$ satisfies the following assumptions:  
\begin{list}{}{}
	\item[$(V')$] $V(x)\in C\left(\mathbb{R}^{N},\mathbb{R}\right)$, $\inf_{x\in \mathbb{R}^{N}}V(x)\geq V_{0}>0$ and $\lim _{|x| \rightarrow \infty} V(x)=\infty$;
	\item[$(f_{1})$] $f \in C(\mathbb{R}^{N}\times \mathbb{R},\R)$ and there exist constants $a_{1}, a_{2} \geq 0, q \in\left[2, \frac{2 N+4 s}{N}\right) $ with $\frac{a_{1}}{2 S_{2}^{2}}+\frac{a_{2}}{q S_{q}^{q}}<\frac{1}{2}$
	such that
	$$
	|f(x, u)| \leq a_{1}|u|+a_{2}|u|^{q-1}, \quad \forall  (x, u) \in \R^{N} \times \R,
	$$
	where $S_{q}$ is the best constant for the embedding of $X \subset L^{q}\left(\R^{N}\right)$ and  $$X=\left\{u \in L^{2}\left(\R^{N}\right): \int_{\R^{N}} \int_{\R^{N}} \frac{|u(x)-u(z)|^{2}}{|x-z|^{N+2 s}} d x d z+\int_{\R^{N}} V(x) u(x)^{2} d x<+\infty\right\};$$
	
	\item[$(f_{2})$]  $\lim\limits_{t\rightarrow \infty}\frac{F(x,t)}{\vert t\vert^{2}}=\infty$ uniformly in $x\in \mathbb{R}^{N}$  and there exists $r_{1}>0$ such that $F(x, u) \geq 0$, for any $x \in \R^{N}$, $u \in \R$ and $|u| \geq r_{1}$, where $F(x,t)=\int_{0}^{t}f(x,s)ds$;
	
	\item[$(f_{3})$] $2F(x,u) < f(x,u)u$, $\forall (x,u)\in \mathbb{R}^{N}\times \mathbb{R}$. 
	
	\item[$(F_{4})$] $f(x,-u)=-f(x,u)$ for all $(x,u) \in \R^N\times \R$
\end{list}
By using the symmetric mountain pass theorem, the authors showed that there exists a constant $\lambda_0> 0$ such that, for any $\lambda \in  (0,\lambda_0)$, problem \eqref{li} possesses infinitely many high energy solutions.

Motivated by these works, in the present paper we are concerned with the existence of infinitely many high energy solutions to the following class of fractional Schr\"{o}dinger equation
\begin{equation} \label{1}
(-\Delta )^{s}u+ V(x)u=f(x,u)+\lambda g(x,u), \quad \text{in}\: \R^N, 
\end{equation}
where $\lambda\in \R$ is a parameter, $V(x)$ satisfies assmptions $(V)$  and $f$ and $g$ satisfy the following assumptions 
\begin{list}{}{}
	\item[$(F_1)$] $f\in C(\R^N\times \R,\R)$ and there exist constants $c_1,c_2>0$ and $p\in (2,2_s^*)$ such that
	\begin{equation*}
	|f(x,u)|\leq c_1|u|+c_2|u|^{p-1},\quad \forall (x,u) \in \R^N\times \R,
	\end{equation*}
	where $2_s^*=\frac{2N}{N-2s}$ is the critical Sobolev exponent.
	\item[$(F_2)$] $\lim\limits_{|u|\rightarrow \infty} \frac{F(x,u)}{u^2}=+\infty$ a.e. $x\in \mathbb{R^{N}}$, where $F(x,u)=\int_0^u f(x,t)dt$  and there exists $r_1>0$ such that
	\begin{equation*}
	\inf_{x\in \mathbb{R}^{N},|u| \geq r_1} F(x,u)\geq 0;
	\end{equation*}
	\item[$(F_3)$] There exist constants $\mu >2$, $c_3>0$ and $a_0>0$ such that such that
	\begin{equation*}
\mu F(x,u) \leq f(x,u)u+c_3 |u|^2,\quad \forall (x,|u|) \in \R^N \times[a_0,\infty).
	\end{equation*}
	
     \item[$(g_1)$] There exist constants $1<\delta_1<\delta_2<2$ and positive functions $\xi_i \in L^{\frac{2}{2-\delta_i}}(\R^N)$ $(i=1,2)$ such that
\begin{equation*}
|g(x,u)|\leq \xi_1(x) |u|^{\delta_1-1} +\xi_2 (x)|u|^{\delta_2-1},\quad \forall (x,u) \in \R^N\times \R.
\end{equation*}
 \item[$(g_2)$] $g(x,-u)=-g(x,u)$ for all $(x,u) \in \R^N\times \R$;
\end{list}
By using the Fountain theorem (i.e. \cite[Theorem 3.6]{Willem}), we prove the existence of an unbounded sequence of nontrivial solutions $\{u_k\}$ to problem (\ref{1}) under assumptions $(V)$, $(F_1)-(F_4)$ and $(g_1)-(g_2)$. Our result extends and sharply improves that in \cite{24}. 

The remainder of this paper is organized as follows. In section 2, we prepare the variational framework of the studied problem. In Section 3, employing the fountain theorem \eqref{223}, we establish the existence of infinitely many high energy solutions to problem \eqref{1}.


\section{Variational setting and main results} 
In this section, for the reader's convince, we shall introduce some notations and we revise some known results about the fractional Sobolev spaces which can be found in \cite{29}. 

As usual, for $1\leq p<+\infty$, we define
\begin{equation*}
\| u\|_{L^p}:=\| u\|_{p}=\left(\int_{\R^N}| u|^{p}dx\right)^{\frac{1}{p}},\quad u\in L^p(\R^N),
\end{equation*}

The fractional Sobolev space $H^{s} (\R^N)=W^{s,2}(\R^N)$ is defined by
\begin{equation*}
H^{s} (\R^N):=\left\{{u\in L^{2}(\R^N)\: :\: \frac{| u(x)-u(y)|}{| x-y|^{\frac{N}{2}+s}}\in L^{2}\left(\R^N\times\R^N\right)}\right\}
\end{equation*}
with the inner product and the norm
\begin{equation*}
\langle u,v\rangle _{H^{s }}=\iint_{\R^N\times \R^N}\frac{(u(x)-u(y))(v(x)-v(y))}{| x-y|^{N+2s}}dxdy+\int_{\R^N}u(x)v(x)dx,
\end{equation*}
\begin{equation*}
\| u\|_{H^{s}}^{2}=\langle u,u\rangle _{H^{s}}=\iint_{\R^N\times \R^N}\frac{| u(x)-u(y)|^{2}}{| x-y|^{N+2s}}dxdy+\int_{\R^N}| u(x)|^{2} dx,
\end{equation*}
where the norm
\begin{equation*}
[u]_{H^{s}}^{2}=\iint_{\R^N\times \R^N}\frac{| u(x)-u(y)|^{2}}{| x-y|^{N+2s}}dxdy
\end{equation*}
is the so called Gagliardo semi-norm of $u$.

Let $\mathscr{S}(\R^N)$ the Schwartz space of rapidly decaying $C^{\infty }$ functions in $\R^N$. We recall that the Fourier transform of a function $u \in \mathscr{S} (\R^N)$ is defined as 
\begin{equation*}
\mathscr{F}u(\xi):= \frac{1}{(2\pi )^{N}}\int_{\R^N}e^{-\mathit{i}x\xi }u (x)dx.
\end{equation*}
By Plancherel's theorem, we have $$\| u \|_{2}=\| \mathscr{F} u \| _{2}, \quad \forall u \in \mathscr{S}(\R^N).$$

Let $s \in (0,1)$, the fractional Laplacian $(-\Delta )^{s}$ of a function $u \in \mathscr{S}(\R^N)$ is defined by means of the Fourier transform as  
\begin{equation*}
\mathscr{F}\left((-\Delta )^{s}u\right) (\xi )=|\xi|^{2s}\mathscr{F}u (\xi ),\quad \forall s \in (0,1).
\end{equation*}
The space $H^{s}(\R^N)$ can also be described via the Fourier transform as follows
\begin{equation*}
H^{s} (\R^N):=\left\{{u\in L^{2}(\R^N)\: :\: \int_{\R^N}\left(1+| \xi |^{2}\right)^{s} |\mathscr{F}u (\xi)|^{2} d\xi <\infty }\right\},
\end{equation*}
and the norm is defined by
\begin{equation*}
\| u\|_{H^{s }}=\left(\int_{\R^N}\left(1+| \xi |^{2}\right)^{s} | \mathscr{F} u (\xi )|^{2} d\xi \right)^{\frac{1}{2}}.
\end{equation*} 

For the problem \eqref{1}, we define the following Hilbert space
\begin{equation*}
H:=\left\{ u \in H^s( \R^N)\: : \: \int_{\R^N} V(x) u(x)^{2} dx <\infty \right\},
\end{equation*}
endowed with the inner product
\begin{equation*}
\langle u,v\rangle:= \langle u,v\rangle_{H} =\iint_{\R^N\times \R^N}\frac{(u(x)-u(y))(v(x)-v(y))}{| x-y|^{N+2s}}dxdy+\int_{\R^N}V(x)u(x)v(x)dx.
\end{equation*}
Then, the norm on $H$ is given by
\begin{equation*}
\|u\|:=\left( \iint_{\R^N\times \R^N}\frac{| u(x)-u(y)|^{2}}{| x-y|^{N+2s}}dxdy+\int_{\R^N}V(x)u(x)^{2} dx\right)^{\frac{1}{2}}.
\end{equation*}
Obviously, by assumptions $(V)$, this norm is equivalent to the standard norm in $H^s( \R^N)$.

From \cite{29}, the embeddings $H^s(\R^N) \hookrightarrow L^p(\R^N)$ is continuous for $p \in [2,2_{s}^{\ast}]$. Therefore,\\ $H \hookrightarrow L^p(\R^N),\: 2\leq p\leq 2_{s}^{\ast}$ is continuous, namely, there exist constants $\eta_p>0$ such that
\begin{equation}\label{6}
\|u\|_{p}\leq \eta_p \|u\|,\quad \forall u \in H,\: p\in [2,2_s^*],
\end{equation}
Moreover, from \cite{teng}, we know that the embedding $H \hookrightarrow L^p(\R^N)$ is compact for $2 \leq  p< 2_s^{*}$ under condition $(V)$. 

For the fractional Schr\"{o}dinger equation (\ref{1}), the associated energy functional is defined on $H$ as follows 
\begin{equation}\label{7}
\begin{split}
I_\lambda(u)=\frac{1}{2}\|u\|^2-\int_{\R^N}F(x,u)dx-\lambda\int_{\R^N}G(x,u)dx.
\end{split}
\end{equation}
By hypotheses $(V)$, $(F_1)$ and $(g_1)$, the functional $I$ is well define and of class $C^1(H,\R)$ with
\begin{equation}\label{8}
\langle I_\lambda^{\prime}(u),v\rangle= \int_{\R^{2N}}\frac{(u(x)-u(y))(v(x)-v(y))}{| x-y|^{N+2s}}dxdy+\int_{\R^N}V(x) u v dx-\int_{\R^N}f(x,u)v dx-\lambda \int_{\R^N}g(x,u)v dx,
\end{equation}
for all $v\in H$. Besides, the critical points of $I$ in $H$ are solutions of problem \eqref{1}. Now, we are ready to state the main result of this paper as follows.

\begin{theorem}\label{main-rslt}
	Assume that conditions $(V)$, $(F_1)-(F_4)$ and $(g_1)-(g_2)$ hold. Then there exists a sequence $\{\lambda_k\}_{k\geq 1} \subset \R^+$ such that $\lim\limits_{k\rightarrow \infty}\lambda_k=\infty$ and problem (\ref{1}) possesses infinitely many nontrivial solutions $\{u_k\}$ provided $|\lambda|\leq \lambda_k$. Moreover, there holds
	\begin{equation*}
	I_\lambda (u_k) \rightarrow \infty \text{ as }  k \rightarrow \infty .
	\end{equation*} 
\end{theorem}

\begin{remark}	
Since the problem (\ref{1}) is defined on the entire space $\mathbb{R}^N$, the main difficulty of this problem is the lack of compactness for Sobolev embedding theorem. In the context of studying of the existence of solutions for the classical Schr\"odinger equation
		\begin{equation*}
		-\Delta u+V(x) u=f(x, u),\quad x \in R^{N},
		\end{equation*}
		Bartsch et al. \cite{Bartsch} presented the general conditions $(V)$ which guarantee the compactness of the embeddings $\left\{u \in H^{1}\left(R^{N}\right)\::\: \int_{R^{N}} V(x) u(x)^{2} d x<+\infty\right\}\hookrightarrow  L^p\left(\R^{N}\right),\: p\in [2,\frac{2N}{N-2}]$. Furthermore, conditions $(V)$ are weaker than the coercivity condition $(V')$ used in \cite{24}.
\end{remark}

\begin{remark}			
Firstly, comparing with Theorem 1.1 in \cite{24}, our assumptions $(F_1)-(F_3)$ are more general than $(f_1)-(f_3)$. Indeed, Let $f(u)=au+b|u|^{p-2}u$, where $a>2S_2^2$, $b > qS_q^q$ and $p\in (2,2_2^*)$. Then, clearly $f$ satisfies $(F_1)$ but not $(f_1)$ since $\frac{a}{2 S_{2}^{2}}+\frac{b}{q S_{q}^{q}}>2$. Secondly, let 
			\begin{equation*}
		f(t)= 3 |t|t-\frac{15}{2} |t|^{1 / 2}t+t,\quad t\in \R.
		\end{equation*}
		Then,
		\begin{equation*}
		F(t)= |t|^{3}-3 |t|^{5 / 2}+\frac{1}{2}t^2.
		\end{equation*}
	It is easy to verify that the above function $f$ satisfies $(F_1)$, $(F_2)$,  $(F_4)$ and $(F_3)$ with $\mu=\frac{5}{2}$. However, $f$ does not satisfy $(f_3)$, in fact we have
		\begin{equation*}
     	f(t)t-2  F(t) = |t|^{3}-\frac{3}{2}|t|^{5 / 2} \leq 0,\quad \forall t \in  \left[-\frac{9}{4},\frac{9}{4}\right].
		\end{equation*}
		This shows that $(f_3)$ is not satisfied. Finally, it is easy to see that $\widetilde{g}(x, u)=\lambda h(x)|u|^{p-2} u$ considered in \eqref{li} is a special case of $g(x, u)$ considered in this paper. Furthermore, unlike \eqref{li}, the parameter $\lambda$ in \eqref{1} is is allowed to be sign-changing. Consequently, Theorem \ref{main-rslt} generalizes and sharply improves Theorem 1.1 in \cite{24}.
	\end{remark}

\begin{remark}	
When $s = 1$, equation \eqref{1} becomes the classical Schr\"odinger equation
		\begin{equation*}
		-\Delta u+V(x) u=f(x, u)+\lambda g(x,u),\quad x \in R^{N},
		\end{equation*}
		As far as we know, our result is new even for the case $s=1$.  

\end{remark}



\section{Proof of the main result}

Hereafter, we shall use $c_i,C_i,\: i=1,2,\ldots$ to denote various positive constants which may change from line to line. We start this section by introducing some variational preliminaries and abstract results that we need to prove our main results.

\begin{definition}\label{PS-seq}[(PS)-condition]
	\begin{itemize}
\item A sequence $\{u_{n}\}\subset H$ is said to be a Palais-Smale sequence at level $c\in \R$ ((PS)$_{c}$ sequence for short) if $I(u_{n})\rightarrow c$ and $I^{\prime}(u_{n})\rightarrow 0$ in $H^*$ the dual space of $H$.
\item The functional $I$ satisfies the Palais-Smale condition at the level $c$ ((PS)$_{c}$ condition for short) if any (PS)$_{c}$ sequence has a convergent subsequence.
	\end{itemize}
\end{definition}   

\begin{lemma}\label{3331}
Under the assumptions of Theorem \ref{main-rslt}, the functional $I_\lambda$ satisfies the (PS)$_{c}$ condition for any $c>0$.
\end{lemma}

\begin{proof}
	Let $\{u_n\}\subset H$ be any (PS) sequence of $I_\lambda$, that is, 
	\begin{equation}\label{PS-C}
	I_\lambda(u_n) \rightarrow c >0 ,\quad I_\lambda'(u_n)\rightarrow 0 \text{ in } H^*.
	\end{equation}
	First, we prove that $\{u_n\}$ is bounded in $H$. Arguing indirectly, suppose that $\|u_n\|\rightarrow \infty$ as $n\rightarrow \infty$. Set $v_n = \frac{u_n}{\|u_n\|}$, then $\|v_n\|=1$, thus $\{v_n\}$ is bounded in $H$. Using assumption $(F_1)$ we have
	\begin{equation}\label{estm-1}
	\begin{split}
	|F(x,u)| =\left|F(x,u)-F(x,0)\right|	& =\left|\int_{0}^{1}f(x,tu)udt\right|\\
	& \leq \int_{0}^{1}\left(c_1|u|^2 t+c_2|u|^p t^{p-1}\right)dt\\
	& =\frac{c_1}{2}|u|^2+\frac{c_2}{p}|u|^p,\qquad \forall (x,u)\in \mathbb{R}^N\times \mathbb{R}.
	\end{split}
	\end{equation}
	Set $\mathcal{F}(x,u_n)=f(x,u_n)u_n-\mu F(x,u_n)$. Therefore, for $x\in\mathbb{R}^{N}$ and $|u(x)| < a_0$, by (\ref{estm-1}), we have
	\begin{equation*}
	\begin{split}
	\left|f(x,u) u- \mu F(x,u)\right| & \leq  \vert f(x,u) u\vert+\mu \vert F(x,u)\vert\\
	& \leq   \left(c_{1}\vert u\vert^{2}+c_{2}\vert u \vert^{p}\right)+\left( c_{1} \frac{\mu}{2}\vert u\vert^{2}+c_{2}\frac{\mu }{p}\vert u\vert^{p}\right)\\
	& \leq  \left(\frac{2+\mu}{2} c_{1}+\frac{p+\mu}{p}c_{2}a_0^{p-2}\right)\vert u\vert^{2}\\
	& = c_{3}\vert u \vert^{2},
	\end{split}
	\end{equation*}
	where $\mu$ and $a_0>0$ are given in $(F_3)$. Combining the above inequality with $(F_3)$, we conclude that there exists $c_4>0$ such that
	\begin{equation}\label{estim-2}
 \mathcal{F}(x,u)=	f(x,u) u- \mu F(x,u)\geq -c_4|u|^2,\quad \forall (x,u)\in \mathbb{R}^{N}\times \mathbb{R}.
	\end{equation} 
	By (\ref{7}), (\ref{8}) and (\ref{estim-2}) , we have
	\begin{equation}\label{I-estm-1}
	\begin{split}
\mu I_\lambda(u_n)-\langle I_\lambda^{\prime}(u_n),u_n\rangle=& \frac{\mu-2}{2}\|u_n\|^2+\int_{\R^N}\mathcal{F}(x,u_n)dx-\lambda\int_{\R^N}\mathcal{G}(x,u_n)dx  \\
	\geq & \frac{\mu-2}{2}\|u_n\|^2 -c_4 \int_{\R^N} |u_n|^2 dx-\lambda\int_{\R^N}\mathcal{G}(x,u_n)dx,
	\end{split}
	\end{equation}
	where $\mathcal{G}(x,u):=g(x,u)u-\mu G(x,u)$. By $(g_1)$ one has
	\begin{equation}\label{g-estm-1}
	\begin{split}
|\mathcal{G}(x,u)|=|g(x,u)u-\mu G(x,u)| &\leq |g(x,u)u|+\mu |G(x,u)|\\
& \leq \xi_1(x)|u|^{\delta_1}+\xi_2(x)|u|^{\delta_2}+\frac{\mu}{\delta_1}\xi_1(x)|u|^{\delta_1}+\frac{\mu}{\delta_2}\xi_2(x)|u|^{\delta_2}\\
&:=\gamma_1 \xi_1(x)|u|^{\delta_1}+\gamma_2\xi_2(x)|u|^{\delta_2},
\end{split}
	\end{equation}
where $\gamma_i=\frac{\mu + \delta_i}{\delta_i}$ ($i=1,2$). Since $\xi_i \in L^{\frac{2}{2-\delta_i}}(\R^N)$, it follows from \eqref{g-estm-1}, the H\"{o}lder inequality and \eqref{6}
\begin{equation}\label{g-estm-2}
\begin{split}
	\left|\int_{\R^N}\mathcal{G}(x,u_n)dx\right| \leq \int_{\R^N} |\mathcal{G}(x,u_n)|dx
	&\leq \gamma_1 \int_{\R^N} \xi_1(x)|u_n|^{\delta_1}dx +\gamma_2 \int_{\R^N} \xi_2(x)|u_n|^{\delta_2}dx\\
	&\leq \sum_{i=1}^2 \gamma_i \left(\int_{\R^N}|\xi_i(x)|^{\frac{2}{2-\delta_i}}dx\right)^{\frac{2-\delta_i}{2}}\left(\int_{\R^N}|u_n|^2dx\right)^{\frac{\delta_i}{2}}\\
	&\leq \gamma_1  \|\xi_1\|_{\frac{2}{2-\delta_1}}\|u_n\|_{2}^{\delta_1}+\gamma_2 \|\xi_2\|_{\frac{2}{2-\delta_2}}\|u_n\|_{2}^{\delta_2}\\
	&\leq  \gamma_1 \eta_2^{\delta_1} \|\xi_1\|_{\theta_1}\|u_n\|^{\delta_1}+\gamma_2\eta_2^{\delta_2} \|\xi_2\|_{\theta_2}\|u_n\|^{\delta_2}\\
	&:= C_1 \|\xi_1\|_{\theta_1}\|u_n\|^{\delta_1}+C_2 \|\xi_2\|_{\theta_2}\|u_n\|^{\delta_2},
\end{split}
\end{equation}
where $C_i= \gamma_i \eta_2^{\delta_i}$ and $\theta_i =\frac{2}{2-\delta_i}$, $i=1,2$.
Combining \eqref{PS-C} with \eqref{I-estm-1} and \eqref{g-estm-2}, for sufficiently large $n\in \N$, there exists a constant $C_3>0$ such that
\begin{equation*}
\begin{split}
C_3 & \geq \mu I_\lambda(u_n)-\langle I_\lambda^{\prime}(u_n),u_n\rangle  \\
\geq & \frac{\mu-2}{2}\|u_n\|^2 -c_4 \int_{\R^N} |u_n|^2 dx-\lambda\int_{\R^N}\mathcal{G}(x,u_n)dx\\
\geq & \frac{\mu-2}{2}\|u_n\|^2 -c_4 \|u_n\|_2^2-|\lambda|\left( C_1 \|\xi_1\|_{\theta_1}\|u_n\|^{\delta_1}+C_2 \|\xi_2\|_{\theta_2}\|u_n\|^{\delta_2}\right),
\end{split}
\end{equation*}	
which yields
\begin{equation*}
\frac{\|u_n\|_2^2}{\|u_n\|^2}\geq \frac{\mu-2}{2c_4}-\frac{1}{c_4}\left[\frac{C_3}{\|u_n\|^2}+|\lambda|\frac{C_1 \|\xi_1\|_{\theta_1}}{\|u_n\|^{2-\delta_1}}+|\lambda|\frac{C_2 \|\xi_2\|_{\theta_2}}{\|u_n\|^{2-\delta_2}}\right].
\end{equation*}	
Since $1<\delta_1<\delta_2<2$ and $\|u_n\|\rightarrow \infty$, we can choose a large $n\in \N$ so that $$\frac{C_3}{\|u_n\|^2}+|\lambda|\frac{C_1 \|\xi_1\|_{\theta_1}}{\|u_n\|^{2-\delta_1}}+|\lambda|\frac{C_2 \|\xi_2\|_{\theta_2}}{\|u_n\|^{2-\delta_2}}\leq \frac{\mu-2}{4c_4},$$
we then conclude
\begin{equation}\label{v-n-L2}
\|v_n\|_2^2=\frac{\|u_n\|_2^2}{\|u_n\|^2}\geq \frac{\mu-2}{4 c_4}>0.
\end{equation}
Set $\Omega_n=\{x \in \mathbb{R}^{N}\: :\: |u_n(x)|\leq r_1 \}$ and $A_n=\{x \in \mathbb{R}^{N}\: :\: v_n(x)\neq 0\}$, then $\meas(A_n)>0$ due to \eqref{v-n-L2}. Besides, since $\|u_n\|\rightarrow \infty$ as $n\rightarrow \infty$, we obtain
\begin{equation}\label{u-n-cvrg}
|u_n(x)|\rightarrow \infty\quad \text{as}\quad  n\rightarrow \infty,\quad  \forall x\in A_n.
\end{equation}
Hence, $A_n\subseteq \mathbb{R}^N \setminus \Omega_n$ for $n\in \mathbb{N}$ large enough. 

Similarly to \eqref{g-estm-2}, by $(g_1)$, \eqref{6} and H\"{o}lder's inequality, we derive that
\begin{equation}\label{g-estm-3}
\begin{split}
\int_{\R^N}  G(x,u_n)dx &\leq \int_{\R^N} \xi_1(x)|u_n|^{\delta_1}dx + \int_{\R^N} \xi_2(x)|u_n|^{\delta_2}dx\\
&\leq \|\xi_1\|_{\theta_1}\|u_n\|_2^{\delta_1}+ \|\xi_2\|_{\theta_2}\|u_n\|_2^{\delta_2}\\
&\leq \|\xi_1\|_{\theta_1}\eta_2^{\delta_1}\|u_n\|^{\delta_1}+ \|\xi_2\|_{\theta_2}\eta_2^{\delta_2}\|u_n\|^{\delta_2}
\end{split}
\end{equation}
Therefore 
\begin{equation}\label{G-estm}
\int_{\mathbb{R}^{N}}\frac{|G(x,u_n)|}{\|u_n\|^2}dx\leq \frac{\|\xi_1\|_{\theta_1}\|u_n\|^{\delta_1}+ \|\xi_2\|_{\theta_2}\|u_n\|^{\delta_2}}{\|u_n\|^2} \longrightarrow 0, \quad \text{ as } n\rightarrow \infty,
\end{equation}
in view of $\|u_n\| \rightarrow \infty$ and $1< \delta_1<\delta_2<2$. Hence, by $(F_1)$, $(F_2)$, \eqref{6}, \eqref{7}, \eqref{PS-C}, \eqref{u-n-cvrg}, \eqref{G-estm} and Fatou’s lemma, we obtain
\begin{equation}\label{ctrd-inqlt}
\begin{split}
0=\lim\limits_{n\rightarrow \infty}\frac{I_\lambda(u_n)}{\|u_n\|^2}
& =\lim\limits_{n\rightarrow \infty}\left[\frac{1}{2}-\int_{\mathbb{R}^{N}}\frac{F(x,u_n)}{\|u_n\|^2}dx-\lambda \int_{\mathbb{R}^{N}}\frac{G(x,u_n)}{\|u_n\|^2}dx\right]\\
&= \frac{1}{2}+\lim\limits_{n\rightarrow \infty}\left[-\int_{\Omega_n}\frac{F(x,u_n)}{u_n^2}v_n^2 dx-\int_{\mathbb{R}^{N}\setminus\Omega_n}\frac{F(x,u_n)}{u_n^2} v_n^2 dx\right]\\
& \leq \frac{1}{2}+\limsup\limits_{n\rightarrow \infty}\left[ \left(c_1+c_2 r_1^{p-2}\right)\int_{\mathbb{R}^{N}} |v_n|^2dx -\int_{\mathbb{R}^{N}\setminus\Omega_n}\frac{F(x,u_n)}{u_n^2} v_n^2 dx\right]\\
& \leq \frac{1}{2}+\left(c_1+c_2 r_1^{p-2}\right)\eta_2^2-\liminf\limits_{n \rightarrow \infty}\int_{\mathbb{R}^{N}\setminus\Omega_n}\frac{F(x,u_n)}{u_n^2} v_n^2 dx\\
& \leq C_4-\int_{A_n}\liminf\limits_{n \rightarrow \infty}\frac{F(x,u_n)}{u_n^2} v_n^2 dx\\
&= C_4- \int_{\mathbb{R}^{N}}\liminf\limits_{n \rightarrow \infty}\frac{F(x,u_n)}{u_n^2}[\chi_{A_n}(x)]v_n^2 dx \longrightarrow -\infty, \quad \text{as}\quad n\rightarrow \infty.
\end{split}
\end{equation}
 This is an obvious contradiction. Consequently, $\{u_n\}$ is bounded in $H$.
	
	Since $\{u_n\}$ is bounded in $H$, then there exists a constant $M>0$ such that 
	\begin{equation}\label{u-n-bnd}
	\|u_n\|\leq M, \quad \forall n\in \N.
	\end{equation}
	Furthermore, passing to a subsequence, there is $u \in H$ such that
	\begin{equation}\label{embed-2}
	\begin{split}
	& u_n \rightharpoonup u  \text{ in } H;\\
	& u_n \rightarrow u\text{ in } L^p(\R^N),\quad 2\leq p<2_s^*;\\
	& u_n \rightarrow u \text{ a.e. in } \R^N.
	\end{split}
	\end{equation}  
	By $(F_1)$, \eqref{6}, \eqref{u-n-bnd}, the H\"older inequality and \eqref{embed-2}, it has
	\begin{equation} \label{f-conv}
	\begin{split}
	\int_{\R^N}\left|f(x,u_n)-f(x,u)\right|(u_n-u) dx & \leq \int_{\R^N} |f(x,u_n)|(u_n-u) dx + \int_{\R^N} |f(x,u)|(u_n-u) dx\\
    & \leq	c_1\int_{\R^N}\left(|u_n|+|u|\right)(u_n-u)dx+c_2 \int_{\R^N}(|u_n|^{p-1}+|u|^{p-1})(u_n-u)dx\\
	&\leq   C_5 \|u_n-u\|_2+C_6\|u_n-u\|_p=o_n(1),
	\end{split}
	\end{equation}
	where $C_5=c_1(\eta_2M+\|u\|_2)$, $C_6=c_2 \left(\eta_p^{p-1}M^{p-1}+\|u\|_p^{p-1}\right)$ and $o_n(1)\rightarrow 0$ as $n\rightarrow \infty$.\\ On the other hand, it follows from $(g_1)$, \eqref{6}, \eqref{u-n-bnd}, H\"older's inequality and \eqref{embed-2} that 
	\begin{equation} \label{g-conv}
	\begin{split}
		\int_{\R^N}|g(x,u_n)-g(x,u)|(u_n-u)dx &\leq \int_{\R^N} |g(x,u_n)|(u_n-u)dx+\int_{\R^N} |g(x,u)|(u_n-u)dx\\
		&\leq  \int_{\R^N} \xi_1(x)|u_n|^{\delta_1-1}(u_n-u)dx + \int_{\R^N} \xi_2(x)|u_n|^{\delta_2-1}(u_n-u)dx\\
		&+ \int_{\R^N} \xi_1(x)|u|^{\delta_1-1}(u_n-u)dx + \int_{\R^N} \xi_2(x)|u|^{\delta_2-1}(u_n-u)dx\\
		&\leq \|\xi_1\|_{\frac{2}{2-\delta_1}}\left(\|u_n\|_{2}^{\delta_1-1}+\|u\|_{2}^{\delta_1-1}\right)\|u_n-u\|_{2}\\
		&+\|\xi_2\|_{\frac{2}{2-\delta_2}}\left(\|u_n\|_{2}^{\delta_2-1}+\|u\|_{2}^{\delta_2-1}\right)\|u_n-u\|_{2}\\
		&\leq (M_1+M_2)\|u_n-u\|_{2}=o_n(1),
	\end{split}
	\end{equation}
where $M_i=	\|\xi_i\|_{\frac{2}{2-\delta_i}}\left(\eta_{2}^{\delta_i-1}M^{\delta_i-1}+\|u\|_{2}^{\delta_i-1}\right),\: i=1,2$.
Then, combining  (\ref{8}), \eqref{PS-C}, \eqref{f-conv} and \eqref{g-conv}, for $n\in \N$ large enough, we have
	\begin{equation*}
	\begin{split}
	o_n(1)& =\langle I_\lambda'(u_n)-I_\lambda'(u),u_n-u\rangle\\
	&= \|u_n-u\|^2-\int_{\R^N}\left[f(x,u_n)-f(x,u)\right](u_n-u) dx-\lambda\int_{\R^N}\left[g(x,u_n)-g(x,u)\right](u_n-u) dx\\
	&=\|u_n-u\|^2+o_n(1).
	\end{split}
	\end{equation*}
Consequently, $u_n\rightarrow u$ strongly in $H$ as $n\rightarrow \infty$. Thus, the functional $I$ satisfies the (PS)$_c$ condition for any $c>0$. The proof is completed. 	
\end{proof}


Let $(X,\|\cdot\|)$ be a Banach space such that $X=\overline{\oplus_{i=1}^{\infty} X_{i}}$ with $\operatorname{dim} X_{i}<+\infty$ for each $i \in \mathbb{N} .$ Set
$$
Y_{k}=\bigoplus_{i=1}^{k} X_{i}, \quad Z_{k}=\bigoplus_{i=k}^{\infty} X_{i}.
$$
In order to prove \autoref{main-rslt}, we shall use the following Fountain Theorem.

\begin{theorem}\label{223}(\cite[Theorem 3.6]{Willem}) 
Let $X$ be an infinite dimensional Banach space. Assume that $\varphi \in C^{1}(X,\mathbb{R})$, $\varphi (-u)=\varphi (u)$ for all $u\in X$. If, for every $k\in \mathbb{N}$, there exist $\rho_k>r_k>0$ such that
\begin{list}{}{}
\item[$(A_1)$]$\varphi$ satisfies the (PS)$_{c}$ condition for every $c>0$;
\item[$(A_2)$] $a_k:=\max\limits_{u\in Y_k,\|u\|=\rho_k}\varphi(u)\leq 0$.
\item[$(A_3)$] $b_k:=\inf\limits_{u\in Z_k,\|u\|=r_k}\varphi(u)\rightarrow +\infty$ as $k\rightarrow \infty$. 
               
\end{list}
Then $\varphi $ has a sequence of critical points $\{u_k\}$ such that $\varphi(u_k)\rightarrow +\infty$.
\end{theorem}

Since $H \hookrightarrow L^{2}\left(R^{N}\right)$ is compact under assumptions $(V)$ and $L^{2}\left(R^{N}\right)$ is a separable Hilbert space, then $H$ possesses is a countable orthonormal basis $\{e_{j}\}_{j=1}^{\infty}$. Define 
\begin{equation*}
X_{j}=\mathbb{R}e_{j},\quad Y_{k}=\bigoplus_{j=1}^{k}X_{j},\quad Z_{k}=\overline{\bigoplus_{j=k+1}^{\infty }X_{j}},\quad k\in \mathbb{Z}.
\end{equation*}
Then, $H=\overline{\bigoplus_{j=1}^{\infty }X_{j}}$ and $Y_{k}$ is finite dimensional.

\begin{lemma}\label{3332}
	Assume that $(V)$, $(F_1)$ and $(g_1)$ hold, then there exist a sequence $\{\lambda_k\}\subset \R^+$ and $r_k>0$ such that 
	\begin{equation*}
	\lambda_k\rightarrow \infty\quad \text{and}\quad \inf\limits_{u\in Z_k,\|u\|=r_k}I_\lambda (u)\rightarrow +\infty\quad \text{as } k\rightarrow \infty
	\end{equation*}
	whenever $|\lambda|\leq \lambda_k$.
\end{lemma}

\begin{proof}
	Similar to Lemma 3.8 in \cite{Willem}, for any $2\leq p< 2_s^*$, we have 
	\begin{equation}\label{25}
	\beta _{k}(p):=\sup\limits_{u\in Z_{k},\Vert u\Vert=1}\Vert u\Vert _{p}\rightarrow 0,
	\end{equation}
	as $k\rightarrow \infty$. 
	
	By \eqref{7}, \eqref{estm-1}, \eqref{g-estm-3} and (\ref{25}) we obtain
	\begin{equation*}
	\begin{split}
	I_\lambda(u)& =\frac{1}{2}\|u\|^{2}-\int_{\mathbb{R}^{N}}F(x,u) dx-\lambda\int_{\R^N} G(x,u)\\
	& \geq  \frac{1}{2}\| u\|^{2}-\frac{c_1}{2} \|u\|_2^2-\frac{c_2}{p} \|u\|_p^{p}-\lambda \left(\|\xi_1\|_{\theta_1}\eta_2^{\delta_1}\|u\|^{\delta_1}+ \|\xi_2\|_{\theta_2}\eta_2^{\delta_2}\|u\|^{\delta_2}\right)\\
	&\geq \frac{1}{2}\| u\|^{2}-\frac{c_1}{2} \beta_k^2(2)\|u\|^2-\frac{c_2}{p} \beta_k^p(p)\|u\|^{p}-|\lambda|\left(\|\xi_1\|_{\theta_1}\eta_2^{\delta_1}\|u\|^{\delta_1}+ \|\xi_2\|_{\theta_2}\eta_2^{\delta_2}\|u\|^{\delta_2}\right).
	\end{split}
	\end{equation*}
	According to \eqref{25}, we can choose a large $k_0>1$ so that
	\begin{equation*}
\beta_k^2(2)	\leq \frac{1}{2c_1},\quad \forall k\geq k_0.
	\end{equation*}   
This provides
	\begin{equation*}
I_\lambda(u)\geq \frac{1}{4}\| u\|^{2}-\frac{c_2}{p} \beta_k^p(p)\|u\|^{p}-|\lambda|\left(\|\xi_1\|_{\theta_1}\eta_2^{\delta_1}\|u\|^{\delta_1}+ \|\xi_2\|_{\theta_2}\eta_2^{\delta_2}\|u\|^{\delta_2}\right).
	\end{equation*}
For any $u\in Z_{k}$ satisfying $\|u\|\geq 1$, we have
\begin{equation*}
\|\xi_1\|_{\theta_1}\eta_2^{\delta_1}\|u\|^{\delta_1}+ \|\xi_2\|_{\theta_2}\eta_2^{\delta_2}\|u\|^{\delta_2}\leq \left( \|\xi_1\|_{\theta_1}\eta_2^{\delta_1}+ \|\xi_2\|_{\theta_2}\eta_2^{\delta_2}\right)\|u\|^{\delta_2},
\end{equation*}	
since $1<\delta_1<\delta_2<2$. Hence, we obtain
\begin{equation}\label{I-k-estm}
\begin{split}
I_\lambda(u) &\geq \frac{1}{4}\| u\|^{2}-\frac{c_2}{p} \beta_k^p(p)\|u\|^{p}-|\lambda|\left(\|\xi_1\|_{\theta_1}\eta_2^{\delta_1}\|u\|^{\delta_1}+ \|\xi_2\|_{\theta_2}\eta_2^{\delta_2}\|u\|^{\delta_2}\right)\\
&\geq \frac{1}{4}\| u\|^{2}-\frac{c_2}{p} \beta_k^p(p)\|u\|^{p}-|\lambda|\left(\|\xi_1\|_{\theta_1}\eta_2^{\delta_1}+ \|\xi_2\|_{\theta_2}\eta_2^{\delta_2}\right)\|u\|^{\delta_2}\\
&= \|u\|^{\delta_2} \left[\frac{1}{4}\| u\|^{2-\delta_2}-\frac{c_2}{p} \beta_k^p(p)\|u\|^{p-\delta_2}-|\lambda|K\right],
\end{split}
\end{equation}
where $K=\|\xi_1\|_{\theta_1}\eta_2^{\delta_1}+ \|\xi_2\|_{\theta_2}\eta_2^{\delta_2}$. 
For each $k\in \N$ sufficiently large, taking
\begin{equation*}
r_k:=\left(\frac{p}{8c_2 \beta _{k}^{p}(p)}\right)^{1/(p-2)}.
\end{equation*}
Then, by virtue of \eqref{25} we obtain 
\begin{equation*}
r_k\rightarrow +\infty\quad \text{as}\quad k\rightarrow \infty.
\end{equation*}
Then, there exists $k_1>1$ such that $r_k\geq 1$ when $k\geq k_1$. By \eqref{I-k-estm}, for $u\in Z_k$, $\|u\|=r_k$, we have
\begin{equation}\label{I-k-estm2}
\begin{split}
I_\lambda(u) & \geq r_k^{\delta_2}\left( \frac{1}{4} r_k^{2-\delta_2}-\frac{c_2}{p} \beta_k^p(p)r_k^{p-\delta_2}-|\lambda|K \right)\\
&=r_k^{\delta_2}\left( \frac{1}{4} r_k^{2-\delta_2}-\frac{c_2}{p} \beta_k^p(p)r_k^{p-\delta_2}-|\lambda|K \right)\\
&=r_k^{\delta_2}\left( \frac{1}{8} r_k^{2-\delta_2}-|\lambda|K \right).
\end{split}
\end{equation}
Choosing $\displaystyle \lambda_k=\frac{r_k^{2-\delta_2}}{16 K}$, then, $\lambda_k \rightarrow \infty$ as $k\rightarrow \infty$ and for any $\lambda \in \R$ satisfying $|\lambda|\leq \lambda_k$ we get from \eqref{I-k-estm2}
\begin{equation*}
I_\lambda(u) \geq \frac{r_k^2}{16}, \quad u\in Z_k,\: \|u\|=r_k .
\end{equation*}	
Hence, for $k\geq \max\{k_0,k_1\}$ we deduce
\begin{equation*}
\inf\limits_{u\in Z_k,\|u\|=r_k}I_\lambda(u)\geq \frac{r_k^{2}}{16} \rightarrow +\infty\quad \text{as}\quad k\rightarrow \infty
\end{equation*}
whenever $|\lambda|\leq \lambda_k$. This completes the proof.
\end{proof} 

\begin{lemma}\label{3333} 
	For any finite dimensional subspace $Y_k \subset H$, there holds
	\begin{equation*}
	\max\limits_{u\in Y_k,\|u\|=\rho_k}I_\lambda(u)\leq 0.
	\end{equation*} 
\end{lemma}

\begin{proof}
	Let $Y_k$ be any finite dimensional subspace of $H$, we claim that there exists a constant $R_k=R(Y_k)>0$ such that $I_\lambda(u) \leq 0$ $\|u\|\geq R_k$. Otherwise, there is a sequence $\{u_n\}\subset Y_k$ such that 
	\begin{equation}\label{26}
	\|u_n\|\rightarrow \infty\quad  \text{ and } \quad  I_\lambda(u_n) \geq 0.
	\end{equation}
	Set $v_n = \frac{u_n}{\|u_n\|}$, then $\|v_n\|=1$. Therefore, by the Sobolev embedding theorem, up to a subsequence, we can assume $v_n \rightharpoonup v$ in $H$, 
	$v_n \rightarrow v$ in $L^p(\R^N)$ ($2\leq p<2_s^*$) and $v_n \rightarrow v$ a.e. in $\R^N$. Set $E=\{x\in \R^N \: :\: v(x)\neq 0\}$. Since on the finite dimensional subspace $Y_k$ all norms are equivalent, there exists a constant $\alpha_k>0$ such that
	\begin{equation*}
	\|u\|_p \geq \alpha_k\|u\|,\quad \forall u\in Y_k,
	\end{equation*}
	and then
	\begin{equation*}
	\int_{\R^N}\left|u_{n}\right|^{p} dx \geq \alpha_k^{p}\left\|u_{n}\right\|^{p}, \quad \forall n \in \mathbb{N},
	\end{equation*}
	which yields
	\begin{equation*}
	\alpha_k^p\leq \lim_{n\rightarrow \infty}\int_{\mathbb{R}^{N}}\left|v_{n}\right|^{p} dx=\|v_n\|_p^p.
	\end{equation*}
	Hence $\meas(E)>0$and then $|u_n(x)|\rightarrow \infty$ for all $x \in E$. 
	Using \eqref{7} and \eqref{26} we obtain
	\begin{equation*}
	\frac{1}{2}\|u_n\|^2\geq \int_{\R^N}F(x,u_n)dx+\lambda \int_{\R^N}G(x,u_n)dx,\quad \forall x\in \R^N.
	\end{equation*}
	Therefore,
	\begin{equation*}
	\frac{1}{2} \geq \int_{\R^N} \frac{F(x,u_n)}{\|u_n\|^2} dx+\lambda \int_{\R^N} \frac{G(x,u_n)}{\|u_n\|^2} dx.
	\end{equation*}
	Then, by \eqref{G-estm} and Fatou's Lemma we deduce
	\begin{equation*}
	\frac{1}{2} \geq \liminf_{n\rightarrow \infty} \int_{\R^N} \frac{F(x,u_n)}{|u_n|^2}|v_n|^2 dx \geq \liminf_{n\rightarrow \infty} \int_{E} \frac{F(x,u_n)}{|u_n|^2}|v_n|^2 dx\geq \int_{E}\liminf_{n\rightarrow \infty} \frac{F(x,u_n)}{|u_n|^2}|v_n|^2 dx=+\infty.
	\end{equation*}  
	We have a contradiction. This shows that there exists a constant $R_k=R(Y_k)>0$ such that $I(u) \leq 0$ for all $u\in Y_k\setminus B_{R_k}(0)$. Hence, choosing $\rho_k>\max\{R_k,r_k\}$, we conclude that 
	\begin{equation*}
	\max\limits_{u\in Y_k,\|u\|=\rho_k} I_\lambda(u) \leq 0.
	\end{equation*} 
\end{proof}

\begin{proof}[\textbf{Proof of Theorem 2.1}]
We have $I_\lambda\in C^1(H,\R)$ is even in view of $(F_4)$ and $(g_2)$. On the other hand, by Lemma \ref{3331} and Lemma \ref{3333}, the functional $I_\lambda$ satisfies the conditions $(A_1)-(A_2)$ of the Fountain Theorem \ref{223}, respectively. Moreover, condition $(A_3)$ is satisfied whenever $|\lambda|\leq \lambda_k=\frac{r_k^{2-\delta_2}}{16 K}$ due to Lemma \ref{3332}. Thus, the functional $I_\lambda$ has a sequence of critical points $\{u_k\}\subset H$ such that $I_\lambda (u_k) \rightarrow \infty$ as $k \rightarrow \infty$, whenever $|\lambda|\leq \lambda_k$, that is, equation \eqref{1} possesses infinitely many solutions.
\end{proof}



\begin{thebibliography}{9}
	


\bibitem{finance} D. Applebaum, L\'{e}vy processes-from probability to finance and quantum groups, Notices Amer. Math. Soc., \textbf{51} (2004), 1336-1347.

\bibitem{Bartsch} T. Bartsch, Z. Q. Wang, M. Willem, Chapter 1-The Dirichlet problem for superlinear elliptic equations, Handbook of Differential Equations Stationary Partial Differential Equation, \textbf{2} (2005), 1-55.


\bibitem{secchi1} B. Bieganowski, S. Secchi, Non-local to local transition for ground states of fractional Schr\"{o}dinger equations on bounded domains, J. Fixed Point Theory Appl., \textbf{22}, 76 (2020). https://doi.org/10.1007/s11784-020-00812-6.

\bibitem{plasma} G. M. Canneori, D. Mugnai, On fractional plasma problems, Nonlinearity, \textbf{31} (2018), 3251-3283.

\bibitem{19} C. Chen, Infinitely many solutions for fractional Schr\"{o}dinger equations in $\R^N$, Electron. J. Differential Equations, \textbf{88} (2016), 1-15.



\bibitem{29} E. Di Nezza, G. Palatucci, E. Valdinoci, Hitchhikers guide to the fractional Sobolev spaces, Bull. Sci. Math., \textbf{136}(5) (2012), 521-573.


\bibitem{18} M. Du, L. Tian, Infinitely many solutions of the nonlinear fractional Schr\"{o}dinger equation, Discrete Contin. Dyn. Syst., \textbf{21}(10) (2016), 3407-3428.	


\bibitem{felmer2} P. Felmer, C. Torres, Radial symmetry of ground states for a regional fractional nonlinear Schr\"odinger equation, Commun.
Pure Appl. Anal., \textbf{13}(6) (2014), 2395–2406.

\bibitem{17} B. Ge, Multiple solutions of nonlinear Schr\"{o}dinger equation with fractional  Laplacian, Nonlinear Anal. (Real World Applications), \textbf{30} (2016), 236-247.

\bibitem{image} G. Gilboa, S. Osher, Nonlocal operators with applications to image processing, Multiscale Model. Simul., \textbf{7}(3) (2008), 1005-1028.

\bibitem{hou} G.L. Hou, B. Ge, J.F. Lu, Infinitely many solutions for sublinear fractional Schr\"odinger-type equation with general potential, Electron. J. Differential Equations, \textbf{97} (2018), 1-13.

\bibitem{khoutir1} S. Khoutir, H. Chen, Existence of infinitely many high energy solutions for a fractional Schr\"{o}dinger equation in $\R^N$, Appl. Math. Lett., \textbf{61} (2016), 156-162.

\bibitem{khoutir2} S. Khoutir, Multiplicity results for a fractional Schr\"odinger equation with potentials, Rocky Mountain J. Math., \textbf{49}(7) (2019), 2205-2226.

\bibitem{Laskin1} N. Laskin, Fractional quantum mechanics and L\'{e}vy path integrals, Phys. Lett. A, \textbf{268}(4) (2000), 298-305.

\bibitem{Laskin2} N. Laskin, Fractional Schr\"{o}dinger equation, Phys. Rev. E, \textbf{66}(5) (2002), 056108.

\bibitem{Laskin3} N. Laskin, Principles of fractional quantum mechanics, Fractional Dynamics, World Sci. Publ., Hackensack, NJ, (2011), 393-427.

\bibitem{24} P. Li, Y. Shang, Infinitely many solutions for fractional Schr\"{o}dinger equations with perturbation via variational methods, Open Math., \textbf{15} (2017), 578-586.

\bibitem{zupei} Z. Shen, Z. Han, Q. Zhang, Ground states of nonlinear Schrödinger equations with fractional Laplacians, Discrete \& Continuous Dynamical Systems-S, \textbf{12}(7) (2019), 2115-2125. 

\bibitem{teng} K. Teng, Multiple solutions for a class of fractional Schr\"odinger equations in $\R^N$, Nonlinear Anal. Real World
Applications, \textbf{21} (2015), 76-86.

\bibitem{mohsen} M. Timoumi,  Infinitely many solutions for fractional Schr\"{o}dinger equations with superquadratic conditions or combined nonlinearities, J. Korean Math. Soc., \textbf{57}(4) (2020), 825-844.

\bibitem{20} Z. Wang, H.S. Zhou, Radial sign-changing solution for fractional Schr\"{o}dinger equation, Discrete Contin. Dyn. Syst., \textbf{36}(1) (2016), 499-508.

\bibitem{26} Q. Wang, Multiple positive solutions of fractional elliptic equations involving concave and convexe nonlinearities in $\R^N$, Commun. Pure Appl. Anal., \textbf{15}(5) (2016), 1671-1688.

\bibitem{Willem} M. Willem, Minimax Theorems, Birkh\"{a}user Boston Inc, Boston, 1996.

\bibitem{zou} W. Zou, Variant fountain theorems and their applications, Manuscripta Math., \textbf{104} (2001), 343–358.

\end{thebibliography}
\end{document}